\documentclass[11pt]{amsart}
\usepackage{amsmath,amsthm,latexsym,xspace,amscd,amssymb}
\usepackage[mathscr]{eucal}
\usepackage[spanish,portuguese]{babel}
\usepackage[latin1]{inputenc}
\usepackage{amsfonts}
\usepackage{amsmath,amsthm}
\usepackage{enumerate}
\usepackage{amssymb}
\usepackage{graphicx}

\usepackage [all]{xy}
\usepackage{longtable}

\usepackage{amsmath,amssymb,amsfonts,textcomp,layout,mathrsfs}

\usepackage[usenames,dvipsnames]{color}

\newcommand{\ao}{\~ ao}

\newcommand{\Ker}{\mathop\mathrm{Ker}\nolimits}
\newcommand{\Img}{\mathop\mathrm{Im}\nolimits}

\newcommand{\dime}{\mathop\mathrm{dim}\nolimits}

\theoremstyle{plain}
\newtheorem{teo}{Theorem}[section]
\newtheorem{propo}[teo]{Proposition}

\newtheorem{lem}[teo]{Lem}
\theoremstyle{definition}
\newtheorem{defin}[teo]{Definition}
\newtheorem{ej}[teo]{Example}

\numberwithin{equation}{section}

\def\ndv{\ {\mid \kern -0.7 em {\scriptstyle \not}} \ \ }
\def\nd{\ {\mid \kern -0.4 em {\scriptstyle \not}} \ \ }

\begin{document}
\begin{center}
    \vspace*{1cm}
    \large{\textsc{About to the  continuity of the application  square root of the  isomorphisms positive in Hilbert spaces}}\\

\vspace*{1.2cm}
\textsc{Jeovanny de Jesus Muentes Acevedo}

    \vskip 1.2cm
	\textsc{Instituto de Matem\'atica e Estatística}

	\textsc{Universidade de S\~ao Paulo}

    \vskip 1.2cm
    \normalsize{S\~ao Paulo, march of 2014}
\end{center}

\[\textbf{Abstract}\]
Let $H$ be a  complex Hilbert space. Every  nonnegative operator
$L \in L (H)$ admits a unique nonnegative square root  $R \in L (H)$, i.e,
 an  nonnegative operator $R \in L (H)$ such that $R^{2} = L$. Let  $GL_{S}^{+}(H)$ be the set of nonnegative  isomorphisms in $L (H)$. First we will prove that $GL_{S}^{+}(H)$ is convex and is a  (real) Banach manifold. Denote by  $L^{1/2}$
the square root nonnegative of the operator $L$. The objective of this paper is to
prove that application  $\mathcal{R}:GL_{S}^{+}(H)\rightarrow GL_{S}^{+}(H),$ defined by $\mathcal{R}(L)= L^{1/2}$,
is a homeomorphism.

\vspace{0,7cm}


\section{Introduction}
Let $ H$ be a complex Hilbert space. The Theorem \ref{raizquadrada}  shows that every nonnegative bounded operator admits a unique nonnegative square root.
The main goal of this paper is to prove that the application that for each nonnegative isomorphism associates its nonnegative square root is  a homeomorphism.

P.M. Fitzpatrick, J. Pejsachowicz and L. Recht  in \cite{fpl}, use this square root to show the existence of a parametrix cogredient to a path of  self-adjoint Fredholm operators, which is used for to define  the spectral flux for this path.
The square root is also used to find the polar decomposition of a bounded operator (see, for example,  T. Kato in \cite{tk}, p. 334). This polar decomposition permits to determine the positive and negative spectral subspaces of any self-adjoint operator.

In the next chapter we will remember some notions as: spectrum of an operator, self-adjoint operator (nonnegative, positive), nonnegative square root of a nonnegative operator, among others. We will also see several known results in functional analysis that will be used in the rest of the work.

We shall denote by $  GL_{S} ^ {+} (H) $ the subset of $ L (H) $ of positive isomorphisms. In the third chapter we will prove that $ GL_ {S} ^ {+} (H) $ is convex (i.e, if $ L $ and $ T \in GL_ {S} ^ {+} (H) $, $ tL + (1 -t) t \in GL_{S} ^ {+}(H)$ for every $ t \in [0,1] $) and is a (real) Banach manifolds.

Let $ L \in L (E) $, where $ E $ is a complex Banach space.
Based on the Cauchy's integral formula, in Chapter 4 we will see that
  if $ f: \Delta \rightarrow \mathbb {C} $ is a holomorphic application, where $ \Delta $ is an open subset of $ \mathbb{C} $ that contains the spectrum of $ L $, denoted by $\sigma(L),$ we can to define the operator $ f (U) \in L (E) $ as \begin{equation}  f (L) = - \frac{1}{2 \pi i} \int_ {\Gamma} f (\lambda ) (L-\lambda I)^{-1} d \lambda , \end{equation}
where $ \Gamma $ is a positively oriented closed path (or a finite number of closed paths which do not intersect)  simple, contained in $ \Delta$ and that contains $  \sigma (L)$ therein.

Finally, in Chapter 5, we prove that if $ L \in L (H) $ is positive, the nonnegative square root  of $ L $ can be expressed in the form $ \gamma (L) $, where $ \gamma: \Theta \rightarrow \mathbb{C} $ is a appropriate   holomorphic  application and $ \Theta $ is an open subset of the complex numbers, which contains the spectrum of $ L $. Using this expression, we shall show that the application square root $ \mathcal{R}:  GL_{S}^{+} (H) \rightarrow GL_ {S} ^ {+} (H) $ is continuous.

\section{Preliminary}
In this chapter we will remember some notions and results that will be of much use at work. Throughout the work, $E$ and $ F $ shall denote complex Banach spaces and $ H $ be a complex Hilbert space with inner product $ \langle \cdot, \cdot \rangle $. Denote by $ L (E, F) $ (or $ L (E) $  if $ F = E$) the Banach space of bounded linear operators $ L: E \rightarrow F $ with the norm
\[\Vert L\Vert =\sup_{x\in E}\frac{\Vert Lx\Vert}{\Vert x\Vert}.\]

We say that an operator $ L \in L (E, F) $ is \textit{invertible} (or is a \textit{isomorphism}) if its inverse, which we denote by $ L^ {-1} $, is limited. If $ F = E $, $ I \in L (E) $ denotes the identity operator.

\begin{propo}\label{aberturagl}Let  $L\in L(E,F)$ be a invertible operator. If $A\in L(E,F)$ and $\Vert A-L\Vert< 1/\Vert L^{-1}\Vert,$ then $A$ is invertible.
\end{propo}
\begin{proof}See, for example, \cite{tk}, p.\ 31.
\end{proof}
We shall denote by $\Img f$ the image of a application $f$.
\begin{propo}\label{inversivelimagem} Suppose that $ L \in L (E, F) $ is injective. Then, $ L^{-1}: \Img L \rightarrow E $ is bounded if and only if there exists $ c> $ 0 such that $ \Vert Lx \Vert \ge c \Vert x \Vert $ for all $ x \in E $. Moreover, if  \, $L^{-1}: \Img L \rightarrow E $ is continuous, then $ \Img L $ is a Banach space.
\end{propo}
\begin{proof}Suppose first that there exists  $ c> $ 0 such that $ \Vert Lx \Vert \ge c \Vert x \Vert $ for all $ x \in E $ and prove that $ L ^ {-1}: \Img L \rightarrow E $ is continuous. Indeed, it is clear that $ L ^ {-1}: \Img L \rightarrow E $ is bijective. If $ y \in \Img L $, we have
\[\Vert y\Vert =\Vert LL^{-1}y\Vert \geq c\Vert L^{-1}y\Vert.\]
That is, $\Vert L^{-1}y\Vert \leq (1/c)\Vert y\Vert $ for all $ y \in \Img L,$ which proves that  $ L ^ {-1 }: \Img L \rightarrow E $ is continuous.

Conversely, if $ L ^ {-1}: \Img L \rightarrow E $ is continuous, there exists $ C> $ 0 such that $ \Vert T ^ {-1} y \Vert \leq C \Vert y \Vert $ for all $ y \in \Img L. $ Since $ L $ is injective, then $ LL^ {-1} Lx = x $ for all $ x \in E $. Hence, $$ \Vert x \Vert = \Vert L ^ {-1} Lx \Vert \leq C \Vert Lx \Vert \quad \text {for all } x \in E. $$  Consequently, there exists $ 1 / C >  0 $ such that $(1 / C) \Vert x \Vert \leq \Vert Lx \Vert $ for all $ x \in E .$

We shall now  that $ \Img L$ is a Banach space. Let $ (y_{n})_{n = 1}^{\infty} $ be a Cauchy sequence in $ \Img L $. So, $ y_{n} = L {x_{n}} $ for a sequence $ (x_ {n}) _ {n = 1} ^ {\infty} $ in $ E $. Now, by hypothesis,
\[\Vert x_{n}-x_{m} \Vert \leq (1 / c) \Vert L (x_{n}-x_{m}) \Vert = (1 / c) \Vert y_{n} -y_{m} \Vert. \]
This fact implies that $ (x_{n})_ {n = 1}^{\infty} $ is a Cauchy sequence. Since $ E $ is Banach, $ (x_ {n}) _ {n = 1} ^ {\infty} $ converges to someone $ x \in E. $ Therefore, by the continuity of $ L $ we have $ (y_ {n}) _ {n = 1} ^ {\infty} $ converges to $ L (x) \in \Img L. $ Hence $ \Img L $ is a Banach space.\end{proof}

\begin{defin}[Spectrum of an operator]\label{ertbbjj} Let $L$ be an operator $L ( E).$
A \textit{regular value} of $ L $ is a number $ \lambda \in \mathbb{C} $ such that the operator $ L - \lambda I $ is a isomorphism. The set of  regular values  $ L $, denoted by  $\rho (L),$ is called \textit{resolvent set} of $ L $. Its complement  $ \sigma ( L ) = \mathbb{C} - \rho (L) $ is called \textit{spectrum} of $ L $.
For $ \lambda \in \rho(L) $, the application  $ R (\lambda) = ( L- \lambda  I ) ^{-1} $  and called \textit{resolvent application} of $ L $.
\end{defin}
It is well known that the spectrum of an operator $ L \in L ( E) $ is a subset not empty $\mathbb{C}$, compact and \begin{equation} \label{dfghjd}  \Vert \lambda \Vert \leq \Vert L \Vert \quad \text { for all } \lambda \in \sigma ( L ) . \end{equation}
As we said in the introduction,  the objective of this paper has to do with the square root of an    nonnegative operator. Remember here in these notions.
\begin{defin}[Self-adjoint operator] Let $ L \in L ( H ) $. We say that $ L $ is \textit{self-adjoint} if \[ \langle Lx , x \rangle = \langle x , Ly \rangle \quad \text{ for all } x , y \in H. \] \end{defin}
A characteristic of the spectrum of self-adjoint operators is given in the following theorem, whose proof can be seen in \cite{at}, p. 324, Theorem 6.11-A.
\begin{teo}\label{wsde12} The spectrum of a self-adjoint operator $ L \in L ( H ) $ is contained in the interval $ [ m , M ] \subseteq \mathbb{R} $, where   \[ m = \inf_{\Vert x \Vert = 1} \langle Lx , x \rangle \quad \text{ and } \quad m = \sup_ {\Vert x \Vert = 1 } \langle Lx , x \rangle . \]
\end{teo}

\begin{defin}Let $ L \in L (H) $ be a self-adjoint operator.
If $ \langle Lx, x \rangle \geq $ 0  for all $ x \in H $, we say that $ L $ is \textit{nonnegative}. If $ \langle Lx, x \rangle> 0,$ $ (\langle Lx, x \rangle <0) $ for all $ x \in H$ with $ \Vert x \Vert = 1 $, we say that $ L $ is \textit{positive}.\end{defin}

It is not difficult to see that if $ L$ is positive, then $ L $ is injective.

\begin{defin}[Square root] Let $ L \in L (H) $ be a nonnegative operator. A \textit{square root} of $ L$ is an operator $ R \in L (H) $ such that $ R^ {2} = L $.\end{defin}

Of course, if $ R $ is a square root of a nonnegative operator $ L $, then $ R $ is also a square root of $L $. In general, a nonnegative operator may have several square roots.
The following theorem, whose proof can be found, for example, in \cite{ek}, p. 476, Theorem 9.4-2, shows that each nonnegative operator has exactly one nonnegative square root.

\begin{teo}[Nonnegative square root]\label{raizquadrada}Each nonnegative operator $ L \in L (H) $ has a nonnegative square root $ R $, which is unique. The operator $ R $ commutes with each operator in $ L (H)$ that commute with  $ L$.
\end{teo}
The demonstration of the previous theorem, given by Kreyszig in \cite{ek}, is to prove that the sequence $ (R_ {n}) _ {n = 1} ^ {\infty} $, where $ R_{0} = 0 $ and
\begin{equation} \label{22} R_ {n +1} = R_ {n} + \frac{1}{2} (L-R_ {n}^{2}), \quad n = 0,1 , 2, ...., \end{equation}
 converges pointwise to nonnegative  operator $R$ such that  $ R^ {2} = L, $ this is, the nonnegative square root of  $L$. However, the theorem is not an explicit form of the operator $ R $.

In the case where $ H$ has finite dimension, it is not difficult to find the nonnegative  square root of a  nonnegative operator, as we will see in the following example.

\begin{ej}Suppose that $ \dime H = n <\infty $ and let $ L \in L (H) $ be a nonnegative operator. Since $ L $ is self-adjoint, there exists an ortonormal basis $ \{v_{1}, v_ {2}, ..., v_ {n} \} $ of $ H $ and complex numbers $ \lambda_ {1} , \lambda_ {2}, ..., \lambda_ {n} $ (not necessarily different) such that
\begin{equation}\label{33tff}Lv_{i}=\lambda_{i}v_{i}\quad \text{ for }i=1,2,...,n.\end{equation}
Since $ L $ is nonnegative, then the $ \lambda_ {1}, \lambda_ {2}, ..., \lambda_ {n} $ are nonnegative real numbers. Take $ R \in L (H) $, defined by
\[Rv_ {i} = \sqrt {\lambda_ {i}} v_ {i} \quad \text{for} i = 1,2, ..., n. \]
So, is $ R $ nonnegative and $ R^ {2} = L $, ie, $ R $ is the nonnegative square root of $ L. $
\end{ej}

When $ H $ has infinite dimension, we  cannot  always have the expressed in the equation \eqref{33tff}. However, in Chapter 5 we will give an explicit formula for the square root for positive definite isomorphism $ L $, that is not simple though, serve to reach our goal.

\medskip

We shall end this chapter  showing that the resolvent application is holomorphic.
\begin{defin}[Applications holomorphic] Let $ \Delta $ be an open subset of $ \mathbb{C} $. We say that an application $f: \Delta \rightarrow E $ is \textit{holomorphic} on $ \lambda_ {0} \in \Delta $ if there exists  the limit
\[\lim_ {\lambda \rightarrow \lambda_ {0}}
\frac{f (\lambda)- f (\lambda_{0})} {\lambda -\lambda_ {0}} = {f ^ \prime} (\lambda_ {0}). \]
If $ f $ is holomorphic at each point of $ \Delta $, we say that $ f $ is \textit{holomorphic} on $ \Delta $ or simply that $ f $ is \textit{holomorphic}.
\end{defin}

\begin{propo}\label{3r3eff}Let $ L \in L (E) $. The application $ R: \rho (L) \rightarrow L (E), $ defined by $ R (\lambda) = (L-\lambda I) ^ {-1} $ for $ \lambda \in \rho (L ) ,$ is holomorphic.
\end{propo}
\begin{proof}It is easy to see that $ R $ is continuous, because is composition  of continuous  applications. Let us fix $ \lambda_ {0} $ in $ \rho (L). $ Note that if $ L $ and $ T \in L (E, F) $ are invertible, then
\begin{equation}\label{t34tgty12}L^{-1}-T^{-1}=-L^{-1}(L-T)T^{-1}.
\end{equation}
In fact,
\begin{align*}L^{-1}-T^{-1}&=L^{-1}L(L^{-1}-T^{-1})TT^{-1}\\
&=L^{-1}(I-LT^{-1})TT^{-1}\\
&=-L^{-1}(L-T)T^{-1}.\end{align*}
Thus, if $\lambda\in \rho(L),$ by \eqref{t34tgty12} we have
\[(L-\lambda I)^{-1}-(L-\lambda_{0} I)^{-1}=-(L-\lambda I)^{-1}(\lambda_{0} I-\lambda I)(L-\lambda_{0} I)^{-1},\]
that is,
\[\frac{(L-\lambda I)^{-1}-(L-\lambda_{0} I)^{-1}}{(\lambda-\lambda_{0})}=(L-\lambda I)^{-1}(L-\lambda_{0} I)^{-1}.\]
Consequently,
\[\lim_{\lambda\rightarrow \lambda_{0}}\frac{(L-\lambda I)^{-1}-(L-\lambda_{0} I)^{-1}}{(\lambda-\lambda_{0})}=(L-\lambda_{0} I)^{-2},\]
which proves the proposition.
\end{proof}
\section{Some topological properties of $GL_{S}^{+}(H)$}
Recall that we denote by $  GL_{S} ^ { + } ( H ) $ the set of positive isomorphisms in $ L ( H ) $. In this section we shall show that $ GL_ { S } ^ { + } ( H ) $ is convex and, moreover, is a  (real)  Banach  manifold.

\medskip

First, assume that $ H $ is  finite-dimension. Take $ L $, $ T \in  GL_{ S } ^ { + } ( H ) $ and $ t \in [ 0,1] $. Note that $ tL + ( 1 - t ) T $ is positive, because if $ x \neq 0 \in H $,
\begin{align*}\langle tL+(1-t)Tx,x\rangle &=\langle tLx,x\rangle +\langle (1-t)Tx,x\rangle\\
&=t\langle Lx,x\rangle +(1-t)\langle Tx,x\rangle\\
&>0. \end{align*}
Thus, $ tL + ( 1 - t ) T$ is injective. Since $ H $ is finite-dimensional, $ tL + ( 1 - t ) T $ is an isomorphism. Therefore $ tL + ( 1- t) T\in GL_ { S } ^ {+} ( H) $. Hence
$GL_  { S } ^ { + } ( H ) $ is a convex subset of $ L ( H ) $.

To show that $ GL_{S} ^ { + } ( H ) $ is convex in the case where $ H $ is infinite-dimensional, first let's look at the following
 known result in the theory of operators in spaces of
Hilbert, which we give a proof for reasons of precision.
\begin{lem}\label{asdfgh33} If $ L: H \rightarrow H $ is a positive isomorphism, then there exists $ c> 0$ such that
\[\inf_{\Vert x\Vert=1}\langle Lx,x\rangle \ge c.\]
\end{lem}
\begin{proof}Let $ x, y \in H $ be such that $ \langle Lx, y \rangle \neq 0. $ For $ a \in \mathbb{R}, $ we take $ w_{a} = x + a \langle Lx, y \rangle y. $ Since $ L $ is positive, we have
\begin{align*}0&\leq\langle Lw_{a}, w_{a}\rangle\\
&=\langle L(x+a\langle Lx, y\rangle y), x+a\langle Lx, y\rangle y\rangle\\
&=\langle Lx,x\rangle+\langle Lx,a\langle Lx, y\rangle y\rangle+\langle L(a\langle Lx, y\rangle y),x\rangle + \langle L(a\langle Lx, y\rangle y), a\langle Lx, y\rangle y\rangle\\
&=\langle Lx,x\rangle+a\overline{\langle Lx, y\rangle}\langle Lx, y\rangle+a\langle Lx, y\rangle\langle L y,x\rangle + a^{2}\langle Lx, y\rangle\overline{\langle Lx, y\rangle}\langle Ly,  y\rangle\\
&=\langle Lx,x\rangle+a\overline{\langle Lx, y\rangle}\langle Lx, y\rangle+a\langle Lx, y\rangle\overline{\langle Lx, y\rangle} + a^{2}\langle Lx, y\rangle\overline{\langle Lx, y\rangle}\langle Ly,  y\rangle.\end{align*}
Thus, for all $a\in \mathbb{R}$, $$\langle Lx,x\rangle+2a\overline{\langle Lx, y\rangle}\langle Lx, y\rangle+ a^{2}\langle Lx, y\rangle\overline{\langle Lx, y\rangle}\langle Ly,  y\rangle\geq0.$$ Consequently, taking the left half of the earlier inequality as a polynomial in $ a $, we have that the discriminant of this polynomial  is less  or equal than 0, that is,  \begin{equation}\label{4rr4}(2\overline{\langle Lx,y\rangle}\langle Lx, y\rangle)^{2}-4\overline{\langle Lx,y\rangle}\langle Lx, y\rangle\langle Ly,  y\rangle\langle Lx,x\rangle\leq0.\end{equation}
Since $\langle Lx, y\rangle\neq 0$, $\overline{\langle Lx,y\rangle}\langle Lx, y\rangle>0$. Therefore, by the inequality  \eqref{4rr4}, we have $\overline{\langle Lx,y\rangle}\langle Lx, y\rangle-\langle Ly,  y\rangle\langle Lx,x\rangle\leq0$, that is, \begin{equation}\label{dsac}\overline{\langle Lx,y\rangle}\langle Lx, y\rangle\leq\langle Ly,  y\rangle\langle Lx,x\rangle.\end{equation}

Is clear that the previous inequality also holds when $ \langle Lx, y \rangle = $ 0 because $ L $ is positive, hence holds for all $ x, y \in H $. Thus, taking $ x \in H $, with $ \Vert x \Vert = 1 $, and $ y = Lx $, by \eqref{dsac} we have $$\Vert Lx\Vert^{4}\leq\langle L^{2}x,  Lx\rangle\langle Lx,x\rangle\leq \Vert L^{2}x\Vert\Vert  Lx\Vert\langle Lx,x\rangle\leq \Vert L\Vert\Vert Lx\Vert^{2}\langle Lx,x\rangle,$$
 that is, \begin{equation}\label{dsac2}\Vert Lx\Vert^{2}\leq \Vert L\Vert\langle Lx,x\rangle\quad\text{for all }x\in H \text{ with }\Vert x\Vert =1.\end{equation}

Since $L$ is a isomorphism, it follows from Proposition \ref{inversivelimagem} that there exists  $c_{1}>0$ such that
 \[ c_{1}\leq\inf_{\Vert x\Vert=1}\Vert Lx\Vert^{2}.\]
Taking  $c=c_{1}/\Vert L\Vert$, by \eqref{dsac2} we have
 \[c\leq\inf_{\Vert x\Vert=1}\langle Lx,x\rangle,\]
witch proves the lemma.
\end{proof}

\begin{teo}\label{definidopositivo} The set $GL_{S}^{+}(H)$ is convex.
\end{teo}
\begin{proof}
Let $L$ and $T$ be  positive isomorphisms and $t\in [0,1]$. Since $tL+(1-t)T$ is  positive, we have  $$\Ker (tL+(1-t)T)=\{0\}.$$

Now,  we shall show that  $tL+(1-t)T$ is surjective. To this end, let us first see that  $\Img tL+(1-t)T$ is closed.
By the previous lemma, we have that there exists positive real numbers $c_{1}$ and $c_{2}$ such that
\[c_{1}\leq\underset{x\in H}{\inf_{\Vert x\Vert=1}}\langle Lx,x\rangle \quad\text{and}\quad c_{2}\leq \underset{x\in H}{\inf_{\Vert x\Vert=1}}\langle Tx,x\rangle .\]
Therefore, if $x\in H$ with $\Vert x \Vert=1,$ by the Cauchy-Schwarz inequality we obtain that
\begin{align*}\Vert (tL+(1-t)T)x\Vert&\geq\langle (tL+(1-t)T)x,x\rangle\\ & = t\langle Lx,x\rangle+(1-t)\langle Tx,x\rangle \\ &\ge (tc_1 +(1-t) c_2)\Vert x\Vert .\end{align*}
Hence,  since $tc_1 +(1-t) c_2>0,$ it follows from Proposition \ref{inversivelimagem}  that  $\Img(tL+(1-t)T)$ is closed.
Thus, since $tL+(1-t)T$ is self-adjoint and knowing the fact that for any operator $S\in L(H)$, $[\Ker S]^{\perp}=\overline{\Img(S^{\ast})}$, we obtain that
\begin{align*}\Img (tL+(1-t)T)&=\overline{\Img (tL+(1-t)T)}\\
&=\overline{\Img (tL+(1-t)T)^{\ast}}\\
&=[\Ker(tL+(1-t)T)]^{\perp}\\
&= H,\end{align*}
that is, $tL+(1-t)T$ is surjective.
Consequently, $tL+(1-t)T$ is a positive isomorphism.\end{proof}

We denote by $L_{S}(H)$ the space of self-adjoint operators in $L(H).$ Note that $L_{S}(H)$ is not vector subspace $L(H)$, because if $L\neq 0\in L_{S}(H)$, $iL\notin L_{S}(H)$. However, it is not difficult to prove that  $L_{S}(H)$ is a vector space over the field of real numbers. Whereas $L_{S}(H)$ whit the topological subspace structure of $L(H)$,  is easy to see that it is a real Banach space.

We shall now show that $GL_{S}^{+}(H)$ is an open subset of $L_{S}(H)$.
\begin{propo}\label{r34rfd555454} The set $GL_{S}^{+}(H)$ is  open in $L_{S}(H).$
\end{propo}
\begin{proof}
Let $L\in GL_{S}^{+}(H)$. If follows from Lemma \ref{asdfgh33}   that there exists $c>0$ such that
\[c\le\inf_{\Vert x\Vert=1}\langle Lx,x\rangle .\]
Let $T\in L_{S}(H)$ be such that $\Vert L-T\Vert <\min\{1/\Vert L^{-1}\Vert, c\}.$ Thus,  $T$ is an isomorphism (see  Lemma \ref{aberturagl}) and, furthermore, for  $x \neq 0\in H$, we have
\[\langle Lx,x\rangle\geq c\langle x,x\rangle >\Vert L-T\Vert \langle x,x\rangle = \langle\Vert L-T\Vert x,x\rangle \geq \langle (L-T) x,x\rangle,\]
that is,
 \[0<\langle Lx,x\rangle-\langle (L-T) x,x\rangle=\langle T x,x\rangle\quad\text{for all }x\in H\text{ with }x\neq 0.\]
So,  $T$ is a positive isomorphism. \end{proof}

Since $L_{S}(H)$ is a  real Banach space,  it follows from Proposition \ref{r34rfd555454} that $GL_{S}^{+}(H)$  is a differentiable Banach  manifold, with
\begin{equation}\label{we133}T_{L}GL_{S}^{+}(H)\cong L_{S}(H)\quad \text{ for }L\in GL^{+}_{S}(H),
\end{equation}
where $T_{L}GL_{S}^{+}(H)$ denotes the  tangent space of $L\in GL_{S}^{+}(H)$.

\section{Functions of operators}
Let  $L\in L(E)$ be fixed. Based on the Cauchy's integral formula, in this chapter we will see that if $ f: \Delta \rightarrow \mathbb{C} $ is a holomorphic application, where $ \Delta $ is an open subset of $ \mathbb{C} $ containing $ \sigma (L), $ we can define the operator $ f (L) \in L (E) $ as \[f (L) = - \frac{1}{2 \pi i} \int_ {\Gamma} f (\lambda ) (L-\lambda I) ^ {-1} d \lambda, \]
where $ \Gamma $ is a path (or a finite number of paths that do not intersect) closed, simple, positively oriented, contained in $ \Delta $ and containing $ \sigma(L) $ in its interior.

In Chapter 5 we shall prove that if $ L \in GL_{ S}^{+} (H) $, the nonnegative square root  of $ L $ can be expressed in the form $ \gamma (L) $, where $ \gamma : \Theta \rightarrow \mathbb {C} $ is a appropriated holomorphic application and $ \Theta $ is an open subset of the complex numbers containing $ \sigma (L) $.

\medskip

We shall first  remember some concepts such as: rectifiable paths (closed, single or positively oriented), Riemann's integral of holomorphic applications, among others.
\begin{defin}Let $\Gamma:[a,b]\rightarrow \mathbb{C}$ be a path, that is, a continuous application.  For any partition of $[a,b]$, given by $P=\{t_{0},t_{1},...,t_{m}\}$, we define
\[\label{comprimento}\Lambda_{\Gamma}(P)=\sum_{k=1}^{m}\Vert \Gamma(t_{k})-\Gamma(t_{k-1})\Vert.\]
If
\[\Lambda_{\Gamma}=\sup\{\Lambda_{\Gamma}(P):P\text{ is a partition of }[a,b]\}\]
is finite, thus we say that $\Gamma$ is rectifiable and its  \textit{length} is $\Lambda_{\Gamma}$.
\end{defin}

\begin{defin}\label{defincurva}Let  $\Gamma:[a,b]\rightarrow \mathbb{C}$ be a path.

We say that $\Gamma$ is \textit{closed} if $\Gamma(a)=\Gamma(b).$
In this case the \textit{interior} of $\Gamma$, denoted by $\mathring{\Gamma}$, is a region of $\mathbb{C}$ bounded by  $\Gamma.$

We say that $\Gamma$ is \textit{simple} if $\Gamma(t_{1})\neq \Gamma(t_{2})$ for $t_{1},t_{2}\in [a,b],$ with $t_{1}\neq t_{2}$ and at least one of them is an interior point $[a,b].$
If $t_{1}<t_{2},$ we use the notation $\Gamma(t_{1})< \Gamma(t_{2})$, whenever $\Gamma(t_{1})\neq \Gamma(t_{2})$.
For simplicity, the image  $\Gamma([t_{1},t_{2}])$ is denoted  by $[\Gamma(t_{1}),\Gamma(t_{2})]$. Furthermore,  we shall write   $\lambda\in\Gamma$ to denote that  $\lambda$ belongs at  $\Gamma([a,b])$.

A closed path $\Gamma$ is \textit{positively oriented} if it traversed in a counterclockwise direction, ie, its interior is on the left, go to the $\Gamma$.

For simplicity, a closed, simple,   positively oriented and rectifiable path $\Gamma:[a,b]\rightarrow \mathbb{C}$ is called  \textit{closed path}, and  moreover, it  will be denoted by  $\Gamma$.\end{defin}

Now recall the definition of the integral of an application  along  a path contained in the complex plane.  Consider  a rectifiable path $\Gamma:[a,b]\rightarrow \Delta$. A \textit{partition}  of  $\Img\Gamma$ is a subset  $P=\{\lambda_{0}, \lambda_{1}, \lambda_{2},...,\lambda_{n}\}\subseteq\Img\Gamma,$ where $\lambda_{0}=\Gamma(a)$ and $\lambda_{n}=\Gamma(b)$, such that
\[\lambda_{0}< \lambda_{1}< \lambda_{2}<...<\lambda_{n}.\] 
The \textit{norm} of the partition $P$ of $\Img\Gamma$ is defined by $$\label{normaparticao}\Vert P\Vert =\max\{ |\lambda_{i}-\lambda_{i-1}|:i=1,...,n\}.$$ 

Now take one application   $f:\Delta \rightarrow E$ and a rectifiable  path $\Gamma:[a,b]\rightarrow \Delta$.  For a partition $P=\{\lambda_{0}, \lambda_{1}, \lambda_{2},...,\lambda_{n}\}$ of the image of $\Gamma$,
let $Q=\{\zeta_{1},\zeta_{2},...,\zeta_{n}\}$, where  $\zeta_{i}\in[\lambda_{i-1},\lambda_{i}],$ for $i=1,2,...,n.$ 
Consider the sum  \begin{equation}\label{mlok}S(P,Q,f)=\sum_{i=1}^{n}(\lambda_{i}-\lambda_{i-1})f(\zeta_{i}).
\end{equation}
\begin{defin}[Integrable applications]\label{3ewdf}We say that $f:\Delta \rightarrow E$ is \textit{integrable} in the path $\Gamma$ if there exists a number $A$ with the following property: For given  $\varepsilon >0,$ there exists a $\delta>0$ such that 
\[|S(P,Q,f)-A|<\varepsilon\text{ whenever }\Vert P\Vert <\lambda.\]
The number  $A$ is called \textit{integral} of $f$ on  $\Gamma$ and we shall denote by  \[
\int_{\Gamma}f(\lambda)d\lambda.\]
\end{defin}
In the case of real continuous functions of a real variable, the analog notion of integral above is equivalent to the classical definition of the Riemann integral, given with the upper and lower sums.
\begin{teo}\label{343rrt}If $f:\Delta \rightarrow E$ is continuous thus is integrable in any rectifiable path  contained in $\Delta.$
\end{teo}
We can to find a prove of the above Theorem  when   $E=\mathbb{C}$, for example, in  \cite{kk}, Chapter 3, \S 9. However, it  is not difficult to see that the same proof also applies to the case where $E$ is any complex Banach space.

\medskip

It follows from Theorem \ref{343rrt} and  Definition \ref{3ewdf} that, if $f$ is continuous, for any  sequence   $(P_{n},Q_{n})_{n=1}^{\infty}$ of partitions of $\Gamma$ such that $\underset{n\rightarrow \infty}{\lim}\Vert P_{n}\Vert=0,$ thus \begin{equation}\label{eqwwr23e}\int_{\Gamma}f(\lambda)d\lambda=\lim_{n\rightarrow \infty}S(P_{n},Q_{n},f).\end{equation}

\medskip

A property  of the above  integral  is given in the following lemma (see, for example, in \cite{kk}, p. 45, Theorem 5).
\begin{lem}\label{inyegralcontinua} If $f:\Delta\rightarrow E$ is holomorphic, thus, for any rectifiable path $\Gamma\subseteq \Delta,$ \[\left\Vert\int_{\Gamma}f(\lambda)d\lambda\right\Vert\leq Ml,\]
 where $M=\sup_{\lambda\in \,\Gamma}|f(\lambda)|$ and $l$ is the length  of $\Gamma.$
\end{lem}
We shall below  present some classical results of  complex functions theory that will be used in this section. The following definition is analogous to the \textit{Cauchy integral formula} for holomorphic complex applications (see, for example,  \cite{kk},  p.\ 61).

\begin{defin}\label{453g545}Suppose that $L\in L(E)$ and let $f:\Delta\rightarrow \mathbb{C}$ be a holomorphic application such that $ \sigma(L)\subseteq\Delta$. Let  $\omega \subseteq \mathbb{C}$ be a open set  such that its  boundary consists  of a finite number of closed paths $\Gamma_{1},$...,$ \Gamma_{n}$ and  $$\sigma(L)\subseteq  \omega= \bigcup_{i=1}^{n}\mathring{\Gamma}_{i}\subseteq
\overline{\bigcup_{i=1}^{n}\mathring{\Gamma}_{i}}\subseteq \Delta.$$
The operator $f(L)$ is defined by
\begin{equation}\label{1qazxc2}f(L)=-\frac{1}{2\pi i}\int_{\partial\omega}f(\lambda)(L-\lambda I)^{-1}d\lambda.\end{equation}\end{defin}
The existence of above integral is follows from Theorem \ref{343rrt}, because the application $\lambda\mapsto f(\lambda)(L-\lambda I)^{-1}$ is continuous (Proposition \ref{3r3eff}). Therefore, $f(L)\in L(E).$

\medskip

The following theorem, whose proof can be found, for example, in \cite{msc}, p.\ 136, Theorem 6.12, we see a first property of the operator defined above.
\begin{teo}\label{12wq}Let $ L $ be an operator on $ L (E) $ and $ \Gamma $ be a closed path such that $ \sigma (L) \subseteq \mathring{\Gamma} $. Then, for each positive integer $ k $, we have
\[L^{k} = - \frac{1}{2 \pi i} \int_ {\Gamma} \lambda ^ {k} (L-\lambda I) ^ {-1} d \lambda .\]
\end{teo}
It is well known that if $ L \in L (E) $ and $ f: \mathbb{C} \rightarrow \mathbb{C} $ is a polynomial given by $ f (\lambda) = \sum_ {k = 0} ^ {n} a_ {k} \lambda ^ {k}, $ where $ a_ {0}, a_ {1}, ..., a_ {n} \in \mathbb{C} $, the operator $ f (L) $ $ $ is defined as $$f(L)=\sum_{k=0}^{n}a_{k}L^{k},\quad\text{where }L^{0}=I.$$
As a consequence of the previous theorem we have
\begin{align*}f(L)&=\sum_{k=0}^{n}a_{k}L^{k}\\
&=-\frac{1}{2\pi i}\sum_{k=0}^{n}a_{k}\int_{\Gamma}\lambda^{k}(L-\lambda I)^{-1}d\lambda\\
&=-\frac{1}{2\pi i}\int_{\Gamma}\sum_{k=0}^{n}a_{k}\lambda^{k}(L-\lambda I)^{-1}d\lambda\\
&=-\frac{1}{2\pi i}\int_{\Gamma}f(\lambda)(L-\lambda I)^{-1}d\lambda.
\end{align*}
Thus, $\sum_{k=0}^{n}a_{k}L^{k}$ coincides  with the definition given in the formula \eqref{1qazxc2}.

\medskip

We shall  end this chapter by presenting  the following  properties  of operator $f(L)$,
whose  proofs can be found , for example, in  \cite{msc}, p.\ 138, Lemma 6.15 and p. 139, Theorem 6.17, respectively.
\begin{lem}\label{fghhjjuuyy}Let $L$  be an  operator in $ L(E)$. Suppose that  $f:\Delta \rightarrow \mathbb{C}$ and $g:\Delta \rightarrow \mathbb{C}$ are holomorphic applications  and  $\sigma(L)\subseteq \Delta$. If $h:\Delta \rightarrow \mathbb{C}$ is defined by $h(\lambda)=f(\lambda)g(\lambda),$ then
\[h(L)=f(L)g(L).\]
\end{lem}
\begin{lem}\label{fucniondelespctro}Let $L$  be an  operator in $ L(E)$. If $f:\Delta \rightarrow \mathbb{C}$ is holomorphic in a open neighborhood    of $\sigma(L),$ then
\[\sigma(f(L))=f(\sigma(L)),\]
that is, $\lambda\in \sigma(f(L))$ if and only if $\lambda=f(\zeta)$ for  someone $\zeta\in \sigma(L).$
\end{lem}

\section{Continuity of nonnegative square root application}
As stated in the introduction, we will prove that the application  that for each nonnegative isomorphism $L\in L(H)$ associates  its nonnegative square root  $R\in L(H)$ is a homeomorphism. For this purpose, we shall see that there exists a holomorphic application $\gamma :\Theta \rightarrow \mathbb{C}$, with $\sigma(L)\subseteq \Theta\subseteq \mathbb{C}$, such that $R=\gamma(L).$

In fact, let  $\Theta=\{a+ib\in\mathbb{C}: a>0\}$ and let \begin{equation}\label{funcionraizc}\gamma:\Theta\rightarrow \mathbb{C}, \quad\text{give by }
\lambda\mapsto |\lambda|^{1/2}e^{\frac{i \text{Arg } z}{2}},\end{equation}
where $\text{Arg } z$ denotes  principal argument of $z\in \mathbb{C}.$
We can to see, for example, in \cite{cnc}, p. 64, Example 15, that $\gamma$ is holomorphic in $\Theta$.
It is easy to see that  \begin{equation}\label{funcionraizc2}\gamma(\overline{\lambda})=
\overline{\gamma(\lambda)}\quad\text{and}\quad\gamma(\lambda)^{2}=
\gamma(\lambda)\gamma(\lambda)=\lambda\quad\text{ for all }\lambda\in \Theta.\end{equation}

Let  $L$ be an operator in $GL_{S}^{+}(H)$. It follows from Theorem \ref{wsde12} that $\sigma(L)$ is a subset of positive real numbers. Hence, $\sigma(L)\subseteq \Theta.$ Consequently, as we saw (Definition \ref{453g545}), for a  appropriate path  $\Gamma$, we can to define $$\gamma(L)=-\frac{1}{2\pi i}\int_{\Gamma}\gamma(\lambda)(L-\lambda I)^{-1}d\lambda.$$  We shall now that $\gamma(L)$ is the nonnegative square root of $L$. For this purpose, we shall need the following basic results.
\begin{lem}\label{isomoreee}Let  $L\in L(H)$ be an isomorphism and $R\in L(H)$ be an operator such that  $R^{2}=L$, then $R$ is a isomorphism.
\end{lem}
\begin{proof}It is clear.
\end{proof}
\begin{lem}\label{raizquadrada2}If $L\in L(H)$ is a  nonnegative isomorphism, then $L$ is  positive.
\end{lem}
\begin{proof}We shall  show that $\langle Lx,x\rangle >0$ for each $x\in H$, with $\Vert x\Vert =1.$ Let $R$ be the nonnegative square root of $L$ (Theorem \ref{raizquadrada}). Since $L$ is a isomorphism, $R$ is an isomorphism  (Lemma \ref{isomoreee}). Let $x\in H$, with $\Vert x\Vert =1.$ Since $R^{2}=L$ and $R$ is self-adjoint, we find that
$$\langle Lx,x\rangle = \langle R^{2}x,x\rangle =\langle Rx,Rx\rangle=\Vert Rx\Vert^{2} >0.$$
\end{proof}

\begin{lem}\label{eswqr} Let $L\in L(H)$ be self-adjoint, then
\begin{equation*}\inf_{\Vert x\Vert =1}\langle Lx,x\rangle = \min _{\lambda\in \sigma(L)}\lambda.
\end{equation*}
\end{lem}
\begin{proof}
Let \[\lambda_{0}=\min _{\lambda\in \sigma(L)}\lambda\quad\text{ and } \quad m=\inf_{\Vert x\Vert =1}\langle Lx,x\rangle.\] It follows from Theorem \ref{wsde12} that $\lambda_{0}\geq m.$ Suppose that $\lambda_{0}> m.$ Thus $L-mI$ is a isomorphism and, for all $x\in H$,  with $\Vert x\Vert =1,$
\[\langle (L-mI)x,x\rangle =\langle Lx,x\rangle -\langle m x,x\rangle= \langle Lx,x\rangle - m \geq 0,\]
that is, $L-mI$ is a nonnegative operator. Since $L-mI$ is a nonnegative isomorphism,  $L-mI$ is  positive by  Lemma  \ref{raizquadrada2}. It follows from Lemma \ref{asdfgh33} that there exists $c>0$ such that \[\inf_{\Vert x\Vert =1}\langle (L-mI)x,x\rangle \geq c.\]
However,
\begin{align*}\inf_{\Vert x\Vert =1}\langle (L-mI)x,x\rangle &=\inf_{\Vert x\Vert =1}[\langle Lx,x\rangle-\langle mx,x\rangle]\\
&=\inf_{\Vert x\Vert =1}\langle Lx,x\rangle -m\\
&=m-m=0.\end{align*}
It is contradicting that \[\inf_{\Vert x\Vert =1}\langle (L-mI)x,x\rangle >0.\] Consequently, $\lambda_{0}=m.$
\end{proof}

\begin{teo}Let $L\in GL_{S}^{+}(H)$. There exists a closed path $\Gamma\subseteq \Theta$, with $\sigma(L)\subseteq  \mathring{\Gamma}$,  such that
\begin{equation}\label{3r4ff4gf}L^{1/2}=\gamma(L)=-\frac{1}{2\pi i}\int_{\Gamma}\gamma(\lambda)(L-\lambda I)^{-1}d\lambda,\end{equation}
where $L^{1/2}$ denotes the nonnegative square root of $L.$
\end{teo}
\begin{proof}
It is clear that there are infinitely many  paths $\Gamma$ such that the integral in \eqref{3r4ff4gf} is well defined (it is sufficient to find a closed path  contained in $\Theta $ such that contains $\sigma(L)$ in its interior). The definition  of this integral doesn't depend  of these paths, however, we will choose a path $\Gamma$ that  facilitate the remaining the proof. By Lemma \ref{asdfgh33} there exists
 $c>0$ such that $$\inf_{\Vert x\Vert=1}\langle Lx,x\rangle \ge c.$$
It follows from equation \eqref{dfghjd},  Theorem \ref{wsde12} and Lemma \ref{asdfgh33} that $\sigma(L)\subseteq [c,\Vert L\Vert]$. Then, $$\sigma(L)\subseteq (c/2, \Vert L\Vert +c/2).$$ Take  $\Gamma$ as the circumference (positively oriented)  centered in $\frac{c+\Vert L\Vert}{2}$ (the middle point of $(c/2, \Vert L\Vert +c/2)$) and passing through the points $c/2$ and $\Vert L\Vert +c/2.$
Hence, \[\sigma(L)\subseteq \mathring{\Gamma}\subseteq \Theta.\]
Consequently, the integral in \eqref{3r4ff4gf} is well defined.
We shall now prove that  $\gamma(L)$ is the nonnegative square root of $L$. By the uniqueness  in the Theorem \ref{raizquadrada}, it is sufficient  to prove that $\gamma(L)^{2}=L$ and that $\gamma(L)$ is nonnegative. In fact, since $\gamma(\lambda)\gamma(\lambda)=\lambda$ for all $\lambda\in \Theta,$ by Lemma \ref{fghhjjuuyy} we obtain that  $$\gamma(L)^{2}=\gamma(L)\gamma(L)=L.$$
See now that $\gamma(L)$ is self-adjoint. For this purpose, we will use  the definition of above  integral as the limit  of a sequence  of sums as in  \eqref{mlok} (see  \eqref{eqwwr23e}). In fact, para cada $n\in\mathbb{N}$, take a partition $P_{n}=\{\lambda_{0}, \lambda_{1}, ...,\lambda_{n}\}$ of $\Gamma$  such that, for $k=0,1,...,n,$ $$|\lambda_{k}-\lambda_{k-1}|\rightarrow 0 \quad\text{ when }n\rightarrow \infty.$$ Moreover, for $k=0,1,...,n,$ let $\xi_{k}=\overline{\lambda}_{n-k}.$ Therefore, $\xi_{k}\in \Gamma$ (by definition of $\Gamma$) and, since $\lambda_{0}< \lambda_{1}<...<\lambda_{n}=\lambda_{0},$ then  $\xi_{0}< \xi_{1}<...<\xi_{n}=\xi_{0}.$ Take $x,y\in H.$ Since $[(L-\lambda I)^{-1}]^{\ast}=(L-\overline{\lambda} I)^{-1}$ and $\overline{\gamma(\lambda)}=\gamma(\overline{\lambda})$, we have
\begin{align*}\langle\sum_{k=1}^{n}\gamma(\lambda_{k})(\lambda_{k}-\lambda_{k-1})
&(L-\lambda_{k}I)^{-1} x,y\rangle\\
&=\langle x,\sum_{k=1}^{n}\overline{\gamma(\lambda_{k})}(\overline{\lambda}_{k}-\overline{\lambda}_{k-1})
(L-
\overline{\lambda}_{k}I)^{-1}y\rangle\\
&=\langle x,\sum_{k=1}^{n}\gamma(\overline{\lambda}_{k})(\overline{\lambda}_{k}-\overline{\lambda}_{k-1})
(L-
\overline{\lambda}_{k}I)^{-1}y\rangle\\
&=\langle x,\sum_{k=1}^{n}\gamma(\xi_{n-k})(\xi_{n-k}-\xi_{n-k+1})
(L-
\xi_{n-k}I)^{-1}y\rangle\\
&=\langle x,-\sum_{k=1}^{n}\gamma(\xi_{n-k})(\xi_{n-k+1}-\xi_{n-k})
(L-
\xi_{n-k}I)^{-1}y\rangle\\
&=\langle x,-\sum_{j=1}^{n}\gamma(\xi_{j-1})(\xi_{j}-\xi_{j-1})(L-\xi_{j-1}I)^{-1}y\rangle.
\end{align*}
Therefore,
\begin{align*}\langle\frac{1}{2\pi i}\sum_{k=1}^{n}\gamma(\lambda_{k})(\lambda_{k}&-\lambda_{k-1})(L-\lambda_{k}I)^{-1} x,y\rangle
\\
&=\langle x,\frac{1}{2\pi i}\sum_{j=1}^{n}\gamma(\xi_{j-1})(\xi_{j}-\xi_{j-1})
(L-
\xi_{j-1}I)^{-1}y\rangle.\end{align*}
By \eqref{eqwwr23e},
\begin{align*}\lim_{n\rightarrow \infty}-\frac{1}{2\pi i}\sum_{k=1}^{n}\gamma(\lambda_{k})&(\lambda_{k}-\lambda_{k-1})(L-
\lambda_{k}I)^{-1}
=\gamma(L)\\
&=\lim_{n\rightarrow \infty}-\frac{1}{2\pi i}\sum_{j=1}^{n}\gamma(\xi_{j-1})(\xi_{j}-\xi_{j-1})(L-
\xi_{j-1}I)^{-1}.\end{align*}
Thus  $\langle  \gamma(L)x,y\rangle=\langle x,\gamma(L)y\rangle$ for $x,y\in H.$
This fact proves that  $\gamma(L)$ is self-adjoint.

We shall now prove that $\gamma(L)$ is nonnegative. The Lemma \ref{fucniondelespctro} implies that  $$\sigma(\gamma(L))=\gamma(\sigma(L)).$$ Thus, since $\sigma(L)\subseteq \mathbb{R}^{+}$, then  $\sigma(\gamma(L))\subseteq \mathbb{R}^{+}$. It is a consequence of Lemma  \ref{eswqr} that
\begin{equation*} 0\leq\min_{\lambda\in\sigma(\gamma(L))}\lambda=\inf_{\Vert x\Vert =1}\langle \gamma(L)x,x\rangle, \end{equation*}
that is,   $\gamma(L)$ is nonnegative.
\end{proof}
It follows from Lemma \ref{isomoreee} that, if $L\in GL_{S}^{+}(H),$ then $L^{1/2}\in GL_{S}^{+}(H).$ Then, by Theorem \ref{raizquadrada} we obtain that the application   \begin{align*}\mathcal{R}:GL_{S}^{+}(H)&\rightarrow GL_{S}^{+}(H)\\
L&\mapsto L^{1/2}
\end{align*} is well  defined.
Using the above Theorem, we shall now prove that this  application is a homeomorphism.
\begin{teo}\label{conti23}The application \begin{align*}\mathcal{R}:GL_{S}^{+}(H)&\rightarrow GL_{S}^{+}(H)\\
L&\mapsto L^{1/2}
\end{align*}
is a homeomorphism.
\end{teo}
\begin{proof}
Let $L\in GL_{S}^{+}(H)$ and $c$ be as in the previous theorem. Consider $r=\min \{1/\Vert L^{-1}\Vert,c/2\}$ and \[B(L,r)=\{T\in L_{S}(H):\Vert L-T\Vert <r\}.\] Thus $B(L,r)\subseteq GL_{S}^{+}(H)$ (see  Proposition \ref{r34rfd555454}). Let $T\in B(L,r).$ We shall now prove that \begin{equation}\label{eqwwsss}\sigma(T)\subseteq (c/3,\Vert L\Vert +c/2).\end{equation} Since $T$ is  positive,  $\sigma(T)$ is a subset of the positive real numbers (Theorem \ref{wsde12}). If $\lambda\in \sigma(T),$ then
$$\lambda\leq\Vert T\Vert \leq\Vert L\Vert+\Vert T-L\Vert<\Vert L\Vert +c/2.$$
Now, suppose that $\lambda\in \mathbb{R}$ with $0<\lambda \leq c/3.$ Let $x\in H$ be of norm 1. Since $\Vert Lx\Vert \geq \langle Lx,x\rangle \geq c\Vert x\Vert$,  $$\Vert Tx\Vert \geq \Vert Lx\Vert -\Vert Lx-Tx\Vert \geq c-c/2,$$ that is,  $\Vert Tx\Vert \geq c/2.$
Consequently,
\[\Vert (T-\lambda I)x\Vert \geq \Vert Tx\Vert -\Vert \lambda x\Vert\geq \frac{c}{2}-|\lambda|>0. \]
This fact proves that $\Ker(T-\lambda I)=\{0\}$. It is a consequence of Proposition \ref{inversivelimagem}  that $T-\lambda I$ is closed. Hence,  since  $T-\lambda I$ is self-adjoint, $$\Img (T-\lambda I)=\overline{\Img (T-\lambda I)^{\ast}}=[\Ker(T-\lambda I)]^{\perp}= H.$$ Therefore,  $\lambda\in \rho(T)$.

Let    $\Gamma$ be the circumference (positively oriented)   with center in  the middle point of the interval  $(c/3, \Vert L\Vert +c/2)$ and passing through the points $c/3$ and $\Vert L\Vert +c/2.$ It follows from \eqref{eqwwsss} that \[\sigma(T)\subseteq \mathring{\Gamma}\subseteq \Theta\quad\text{for all }T\in B(L,r).\] Hence, we can to define
$$\gamma(T)=-\frac{1}{2\pi i}\int_{\Gamma}\gamma(\lambda)(T-\lambda I)^{-1}d\lambda\quad\text{ for }T\in B(L,r).$$
The above theorem implies that, for  $T\in B(L,r)$, the  operator $T^{1/2}$ is $\gamma(T)$, that is, $$\mathcal{R}(T)=\gamma(T)\quad\text{ for }T\in B(L,r).$$

We shall prove that $\mathcal{R}$ is continuous using the above equality. 
Let $T\in B(L,r).$
By \eqref{t34tgty12} we obtain that  $$(T-\lambda I)^{-1}- (L-\lambda I)^{-1}=-(T-\lambda I)^{-1}(T-L)(L-\lambda I)^{-1}\quad\text{for }\lambda\in \Gamma.$$ Thus,
\begin{align*}
\gamma(T)-\gamma(L)&=-\frac{1}{2\pi i}\int_{\Gamma}\gamma(\lambda)(T-\lambda I)^{-1}d\lambda + \frac{1}{2\pi i}\int_{\Gamma}\gamma(\lambda)(L-\lambda I)^{-1}d\lambda \\
&=-\frac{1}{2\pi i}\int_{\Gamma}\gamma(\lambda)[(T-\lambda I)^{-1} - (L-\lambda I)^{-1}]d\lambda \\
&=\frac{1}{2\pi i}\int_{\Gamma}\gamma(\lambda)(T-\lambda I)^{-1}(T-L)(L-\lambda I)^{-1}d\lambda.
\end{align*}
Consequently, \[\gamma(T)-\gamma(L)=\frac{1}{2\pi i}\int_{\Gamma}\gamma(\lambda)(T-\lambda I)^{-1}(T-L)(L-\lambda I)^{-1}d\lambda.\]

Let $M>0$ be such that $$\Vert (T-\lambda I)^{-1}\Vert \leq M\quad \text{for }\lambda \in \Gamma \text{ and }T\in B(L,r). $$  Thus
\[\Vert\gamma(\lambda)(L-\lambda I)^{-1}(L-T)(T-\lambda I)^{-1}\Vert\leq m M^{2}\Vert L-T\Vert, \text{ where }m=\max_{\lambda\in \,\Gamma}|\gamma(\lambda)|.\]
Let $\varepsilon >0$ given. If $\Vert L-T\Vert <\varepsilon,$ then
$$\Vert \mathcal{R}(L)-\mathcal{R}(T)\Vert=\Vert\gamma(L)-
\gamma(T)\Vert<\frac{1}{2\pi}mM^{2}l\varepsilon,$$
where $l$ is the length of  $\Gamma$ (Lemma \ref{inyegralcontinua}).
Consequently, $\mathcal{R}$ is continuous.

On the other hand, it is not difficult to show that the application \begin{align*}\mathcal{C}:GL_{S}^{+}(H)&\rightarrow GL_{S}^{+}(H)\\
L&\mapsto L^{2}
\end{align*}
is the inverse of $\mathcal{R}$. Therefore, $\mathcal{R}$ is a homeomorphism.
 \end{proof}

\end{document}